\documentclass[11pt, a4paper]{amsart}
\usepackage{hyperref}
\usepackage{amsmath}
\usepackage{amssymb}
\usepackage{amsthm}
\usepackage{enumerate}
\newcommand{\RR}{\mathbb{R}}

\newcommand{\NN}{\mathbb{N}}
\newcommand{\ZZ}{\mathbb{Z}}

\newcommand{\PP}{\mathbb{P}}

\newcommand{\cA}{\mathcal{A}}
\newcommand{\cB}{\mathcal{B}}
\newcommand{\cC}{\mathcal{C}}

\newcommand{\cP}{\mathcal{P}}
\newcommand{\cU}{\mathcal{U}}

\renewcommand{\l}{\left}

\newcommand{\eff}{\mathrm{eff}}
\newcommand{\Erw}{ \mathbf{E}_{x} }
\newcommand{\Prb}{ \mathbf{P}_{x} }
\newcommand{\Tr}{{\mathop{\mathrm{Tr}}}}
\newcommand{\vol}{{\mathop{\mathrm{vol}}}}

\DeclareMathOperator{\const}{\mathrm{const}}

\newcommand{\dist}{{\ensuremath{\mathrm{dist}}}}
\def\coloneq{\mathrel{\mathop:}=}
\def\eqcolon{=\mathrel{\mathop:}}
\newtheorem{thm}{Theorem}
\newtheorem{lem}[thm]{Lemma}

\newtheorem{ass}[thm]{Assumption}
\theoremstyle{definition}
\newtheorem{dfn}[thm]{Definition}
\newtheorem{rem}[thm]{Remark}
\newcommand{\FinSetsZd}{\mathcal{F}_{\text{fin}}(\ZZ^d)}

\newcommand{\interior}[1]{\overset{\circ}{#1}}
\newcommand{\numb}[2][]{\thinspace{\sharp_{#1}} #2}

\newcommand{\bigpar}{\par\bigskip\noindent\ignorespaces}

\newcommand{\hm}[1]{\leavevmode{\marginpar{\tiny%
$\hbox to 0mm{\hspace*{-0.5mm}$\leftarrow$\hss}%
\vcenter{\vrule depth 0.1mm height 0.1mm width \the\marginparwidth}%
\hbox to 0mm{\hss$\rightarrow$\hspace*{-0.5mm}}$\\\relax\raggedright #1}}}

\begin{document}

\title[$L^p$-approximation of the integrated density of states ]{$L^p$-approximation of the integrated density of states for Schr\"odinger operators with finite local complexity}

\author[M.~J.~Gruber]{Michael J.\ Gruber}
\address[M.G.]{TU Clausthal\\
Institut f\"ur Mathematik\\
38678 Clausthal-Zellerfeld\\
Germany}
\urladdr{http://www.math.tu-clausthal.de/~mjg/}

\author[D.~H.~Lenz]{Daniel H.\ Lenz}
\address[D.L.]{Mathematisches Institut, Friedrich-Schiller-Universit\"at Jena , 07743 Jena, Germany}
\urladdr{http://www.analysis-lenz.uni-jena.de/}

\author[I.~Veseli\'c]{Ivan Veseli\'c}
\address[I.V.]{Emmy-Noether-Projekt `Schr\"odingeroperatoren',
Fakult\"at f\"ur Mathematik,  09107 TU  Chemnitz, Germany}
\urladdr{http://www.tu-chemnitz.de/mathematik/enp}

\thanks{\copyright 2010 by the authors. Faithful reproduction of this article,
        in its entirety, by any means is permitted for non-commercial purposes.}


\begin{abstract}
We study spectral properties of Schr\"odinger operators on $\RR^d$.
The electromagnetic potential is assumed to be determined locally
by a colouring of the lattice points in $\ZZ^d$, with the property that
frequencies of finite patterns are well defined.
We prove that the integrated density of states (spectral distribution function)
is approximated by its finite volume analogues, i.e.~the normalised eigenvalue counting functions.
The convergence holds in the space $L^p(I)$ where $I$ is any finite energy interval and $1\leq p< \infty$ is arbitrary.
\end{abstract}

\maketitle

\section{Introduction}\label{s:intro}

Spectral properties play a key role in the analysis of selfadjoint operators.
This is in particular the case for Hamiltonians describing the time evolution of
quantum mechanical systems.
In the context of mathematical physics one often studies the integrated density of states,
in the following abbreviated by IDS.
It is very natural to think of the IDS as the normalized eigenvalue counting function
of the restriction of the Hamiltonian to a large but finite volume system.
This leads to the question in what sense and how well one can approximate the IDS of the full Hamilton operator
by the spectral distribution functions of appropriately chosen finite-volume analogues.
This question has been pursued for various types of selfadjoint operators, resp.\ Hamiltonians,
in particular in the mathematical physics and geometry literature.
Let us mention the seminal papers \cite{Pastur-71, Shubin-79} and the recent reviews
\cite{KirschM-07,Veselic-07b}. There one can find also an overview of the literature up to '07.

It turns out that in the discrete and in the one-dimensional setting
one can control the convergence of finite volume
approximants to the IDS very well. Let us state this more precisely:
\begin{itemize}\sloppy
\item For difference operators (with finite range) on combinatorial graphs the
eigenvalue counting functions are bounded. This leads to uniform convergence (in
supremum norm) for the IDS \cite{LenzMV-08,LenzV-09}.
\item
For Schr\"odinger operators on metric graphs (so
called quantum graphs) with constant edge lengths
one can achieve uniform
convergence for the IDS as well \cite{GruberLV-07}.
Here one has to assume that the randomness satisfies a finite local complexity condition.
Note that for quantum graphs the eigenvalue counting functions
are unbounded: However, the technically relevant objects are the spectral shift functions, which are still bounded.
\item Metric graphs with non-constant edge lengths lead to unbounded shift
functions; in this case, convergence holds locally uniformly, as well as globally
uniformly with respect to a weighted supremum norm \cite{GruberLV-08}.
\end{itemize}
See also \cite{GruberLV-08} for an overview.

For electromagnetic Schr\"odinger operators on $\RR^d$ even the perturbation by a compactly supported
potential may lead to a locally unbounded spectral shift function.
Thus the shift function diverges not only at infinity but also on compact energy intervals.
This is in particular the case for Landau-type  Hamiltonians, see, e.g.,
\cite{RaikovW-02b,HundertmarkKNSV-06} and references therein.
One may expect that the situation will be better for certain
random perturbations of the  Landau Hamiltonian.
In order to obtain continuity of the IDS of Landau-type Hamiltonians plus a random, ergodic potential
one has to pose appropriate conditions on the randomness.
They amount to regularity conditions on the random distribution (see
\cite{CombesH-96,Wang-97,HupferLMW-01a,CombesHK-07} and references therein).
The results of the present paper apply to highly ``singular'' distributions, though: those of
Bernoulli type;
at each lattice site in $\ZZ^d$, local electric and magnetic potentials are chosen randomly from a finite set of prototypes.
For such models there are no results on the continuity of the IDS.
Thus this property cannot help us proving uniform convergence of the distribution functions.
Our main result is that one can still achieve a strong form of convergence.
More precisely, we show that convergence holds in the space $L^p(I)$ for any finite interval
$I\subset\RR$ and any finite $p$.

In the following section we describe our model and assumptions and state the main theorems.
Section~\ref{s:ssf} provides bounds on the spectral shift function. They are applied in Section
\ref{s:additivity}  which establishes certain almost additivity properties and thus concludes the proof of the main theorem.
The latter is applied to certain types of alloy type random Schr\"odinger operators in the final section.

\section{Model and results}\label{s:results}
Throughout the paper we will consider electromagnetic Schr\"odinger operators which satisfy the following regularity
\begin{ass}\label{a:regularity}
Let $\cU$ be an open set in $\RR^d$, $A\colon \cU\to \RR^d$ a magnetic vector potential, each component of
which is locally square integrable, $V=V_+ -V_-\colon \cU \to \RR  $ a scalar electric potential,
such that its positive part $V_+ \ge 0$ is locally integrable and its negative part $V_- \ge 0$
is in the Kato class.
This implies that $V_-$  is relatively form bounded with respect to $-\Delta^\cU$, the Dirichlet Laplacian on $\cU$, with
relative bound $\delta$ strictly smaller than one.
Under these conditions  the magnetic Schr\"odinger operator
\begin{equation}\label{E:Hdefn}
  H^\cU=(-i\nabla-A)^2+V
\end{equation}
is well defined via the corresponding lower semi-bounded quadratic form with core $C_c^\infty(\cU)$
\cite{Simon-79c}.
Due to this choice of core, we say that $H^\cU$ has Dirichlet boundary conditions.
\end{ass}

Let us mention locally uniform $L^p$-integrability conditions which are sufficient for $V_-$ to be in the Kato-class.
More precisely, if $V_-$ satisfies
\begin{displaymath}
\Vert V_- \Vert_{L^p_{\text{loc},\text{unif}}(\RR^d)}
= \sup_{x\in\RR^d} \Big( \int_{|x-y|\le 1} \vert V_-(y)\vert^p \, dy \Big)^{1/p} <\infty
\end{displaymath}
for $p=1$ if $d=1$ and $p>d/2$ if $d\ge 2$, then it belongs to the Kato-class.

\bigpar
Next we want to introduce the notions of a colouring and a pattern. For this purpose we denote the set of all finite subsets of $\ZZ^d$ by $\FinSetsZd $
and an arbitrary finite set by ${\mathcal A}$. A \textit{colouring} is a map 
$${\mathcal C}: \ZZ^d\rightarrow \mathcal A$$ 
and a \textit{pattern} is a map $P: D(P) \rightarrow \mathcal A$,
where
$D(P)\in \FinSetsZd $ is called the domain of $P$.
We denote the \textit{set of all patterns} by ${\mathcal P}$. For a fixed $Q\in \FinSetsZd $  we denote
the subset of ${\mathcal P}$ which contains only the patterns with domain $Q$ by ${\mathcal P}(Q)$.
Given a set $Q\subset D(P)$ and an element $x\in \ZZ^d$ we define a \textit{restriction of a pattern} $P\vert_Q$ and the
\textit{translate of a pattern} $P+x$ by the vector $x$ in the following way
\begin{equation*}
P\vert_Q\colon Q\rightarrow {\mathcal A}, g\mapsto P\vert_Q(g)=P(g), \quad P+x\colon D(P)+x\rightarrow {\mathcal A}, y+x\mapsto P(y)
\end{equation*}
Two patterns $P_1, P_2$ are \textit{equivalent} if there exists an $x \in \ZZ^d$ such that $P_2= P_1+	x$.
The equivalence class of a pattern $P$ in $\cP$ is denoted by $\tilde P$.
This induces on $\cP$ a set of  equivalence classes $\tilde {\mathcal P} $.
For two patterns $P$ and $P'$ the number of occurrences of the pattern $P$ in $P'$ is denoted by
\begin{equation*}
\numb[P]{(P')}\coloneq  \numb{\{x\in \ZZ^d\,\vert\, D(P)+x\subset D(P'), P'\vert_{D(P)+x}= P+x\}}.
\end{equation*}
Here, $\numb{}$ denotes the cardinality of a finite set.

Next we define the notion of a van Hove sequence and of the frequency of a pattern along a given van Hove sequence.
A sequence  $(U_j)_{j\in \mathbb N}$ of finite, non-empty subsets of $\ZZ^d$ is called a van Hove sequence if
\begin{equation} \label{e:van-Hove}
 \text{ for all } M \in \NN :\quad \lim_{j\to \infty}  \frac{\numb{\partial^{M}U_j}}{\numb{U_j}} =0.
 \end{equation}
Here $\partial^{M}U  = \{x \in U\mid \dist(x, \ZZ^d\setminus U) \le M \} \cup \{x \in \ZZ^d\setminus U\mid \dist(x,U) \le M \}. $

It is sufficient to check the relation \eqref{e:van-Hove} for $M=1$, it then follows for all $M\in \NN$, cf.\ for instance Lemma 2.1 in \cite{LenzSV}.
If for a pattern $P$ and a van Hove sequence $(U_j)_{j\in \mathbb N}$ the limit
\begin{equation*}
\nu_P\coloneq \lim\limits_{j\rightarrow \infty}\frac{\numb[P]{({\mathcal C}\vert_{ U_j})}}{\numb{U_j}}.
\end{equation*}
exists, we call $\nu_P$ the \textit{frequency of $P$  along $(U_j)_{j\in \mathbb N}$} in the colouring ${\mathcal C}$.

\bigpar
In our setting $\cA$ will be  a finite collection of pairs $(a,v)$ where
$a$ is a function $\RR^d \to \RR^d$ such that all its components are in $L^2$ and
$v$ is a function $\RR^d \to \RR$ such that its positive part is in $L^1$ and its negative part is in the Kato class,
cf.\ Assumption~\ref{a:regularity}.
Moreover, both the support of $a$ and of $v$ are contained in  
$$W_0\coloneq[0,1]^d.$$
 Given a colouring $\cC\colon \ZZ^d \to \cA$
we denote by $\cC_a$ its first and by $\cC_v$ its second component.
To each colouring  $\cC$ we associate  an electromagnetic potential $(A_\cC,V_\cC)\colon \RR^d \to \RR^d  \times \RR$
\begin{equation*}
 A_\cC(x) \coloneq  \sum_{k \in \ZZ^d}  \cC_a(k)(x-k), \quad V_\cC(x) \coloneq  \sum_{k \in \ZZ^d}  \cC_v(k)(x-k)
\end{equation*}
and a Schr\"odinger operator
\begin{equation*}
H_\cC\coloneq (-i\nabla-A_\cC)^2+V_\cC.
\end{equation*}
Note that  for any open subset $\cU$ of $\RR^d$ the restriction $H_\cC^\cU$ of $H_\cC$ to $\cU$ with Dirichlet boundary conditions
satisfies Assumption~\ref{a:regularity}.

\bigpar
Now we want to define the IDS of $H_\cC$ and of finite restrictions thereof. For this purpose we need
some more notation.
In order to associate finite subsets of $\ZZ^d$ with bounded subsets of $\RR^d$,
we define
\[
W\colon \FinSetsZd  \to \cB (\RR^d), \quad  Q \mapsto W_Q \coloneq \bigcup_{t\in Q} (W_0+t)
 \]
where we use the natural embedding $\ZZ^d\subset\RR^d$ and denote the Borel-$\sigma$-algebra by $\cB$.
Furthermore, if $\cU$ is an open set in $\RR^d$, $H^\cU$ a Schr\"odinger operator
defined on $\cU$ and $Q\in\FinSetsZd$ such that $\interior{W_Q}\subset\cU$ we
define, by a slight but natural abuse of notation,
\[ H^Q\coloneq H^{ \interior{W_Q} }, \]
i.e.\ the restriction of $H^\cU$ to the interior of $W_Q$ in the sense of
quadratic forms as above.

Let us denote by $\chi_{(-\infty,\lambda]} \big( H_\cC\big)$ and
$\chi_{(-\infty,\lambda]} \big( H_\cC^Q \big)$ the spectral families
of $H_\cC$ and $H_\cC^Q$, respectively.
The IDS of $H_\cC^Q$ is the distribution function
\[
N\big(\lambda, Q\big) \coloneq \Tr \big[ \chi_{(-\infty,\lambda]} \big( H_\cC^{Q} \big)\big],
\]
divided by the volume $\vol W_Q=\numb Q$.
Since $Q$ is finite, $H_\cC^Q$ is an elliptic operator on a bounded domain
and $\chi_{(-\infty,\lambda]} \big( H_\cC^Q \big)$ is trace-class. However, 
$\chi_{(-\infty,\lambda]} \big( H_\cC\big)$ is not, which is the reason why we need
the existence theorem below, and why $H_\cC$ may display any interesting spectral features at all.

For $M\in \NN$ we denote by $C_M\subset \ZZ^d$ the cube at the origin with side length $M  - 1$, i.e.
\begin{equation*}
C_M\coloneq \{ x\in \ZZ^d:  0\leq x_j \leq M-1, j=1,\ldots, d\}.
\end{equation*}

Now we are prepared to state our main theorem:
\begin{thm}\label{t:IDS-colouring}
Let $\cC\colon \ZZ^d \to \cA$ be a colouring,  $(U_j)_{j\in \mathbb N}$ a van Hove sequence such that
for all patterns $P\in \bigcup_{M\in \NN}{\cP}(C_M)$ the frequencies $\nu_P$
exist.
Let $I\subset \RR$ be a finite interval.  Then there exists a function
$N$ belonging to all $ L^p(I)$ with $1\leq p < \infty$   and independent of the van Hove sequence such that
for $j\to \infty$ we have
\begin{equation}
 \int_I \left|N\big(\lambda) - \frac{1}{\numb{U_j}}  N\big(\lambda, {U_j}\big) \right|^p \, d\lambda\to 0
\end{equation}
for any any $p \in [1,\infty)$.
More precisely, for any $M \in \NN$ the above integral is bounded by
\begin{multline*}
\frac{C}M+\left( C(T+C)^{\frac d2}+c_{p,d}C^{\frac1p} \right) \frac{\numb{\partial^MU_j}}{\numb{U_j}}+ \\
+ C(T+C)^{\frac d2} \sum_{P\in\cP(C_M)} \Big \vert \frac{\numb[P]{ ({\mathcal C}\vert_{U_j})}}{\numb{U_j}}- \nu_P  \Big \vert
\end{multline*}
where $T=\sup I$. The dependencies of the constants appearing in the estimate are as follows:
$c_{p,d}$ depends only on the dimension $d$ and the exponent $p$,   and
$C$ depends only on $d$ and on (the Kato norm of) $V_{\cC,-}$.

\end{thm}
\begin{rem}
\begin{itemize}
\item Obviously, the theorem yields a function $N\in L^p(\RR)$ such that on each finite interval $I$ the above
holds, and we may extend this to (upper) semibounded intervals since our operators are uniformly
semibounded below.
\item In particular one may choose the van Hove sequence $U_j\coloneq C_j$ or a similar family of expanding cubes
(if the frequencies exist) so that one gets an explicit $\frac1j$-decay for the second term in the error estimate.
A sequence of cubes  is a common choice for defining the IDS.
\item In the random setting, introduced in Section~\ref{s:random},  there is an alternative definition of the IDS
(Pastur-Shubin formula) which coincides with $N$, as we show there.
\end{itemize}
\end{rem}

\section{Bounds on the spectral shift function of facets}\label{s:ssf}
In this section we prove certain bounds on the spectral shift function (SSF) which are needed in Section \ref{s:additivity}.
They concern the SSFs of two electromagnetic Schr\"odinger operators which differ by an (additional)
Dirichlet boundary condition on a facet.
This difference can be understood as a generalized (positive) compactly supported potential.
Hence we can use results established in \cite{HundertmarkKNSV-06}, either verbatim or
with slight modifications.

Recall that if $V_-$ is in the Kato class, then there
exists some number $\delta$ smaller than one such that $V_-$ is relatively $\Delta$-bounded
with relative bound $\delta$.

We quote Lemma 5 from  \cite{HundertmarkKNSV-06}:

\begin{lem}\label{l:lower}
Let $H^\cU$ be a Schr\"odinger operator defined on the open set $\cU \subset \RR^d$ of finite volume
which satisfies Assumption~\ref{a:regularity}. Denote by $E_n$ the $n^{\text{th}}$ eigenvalue counted from
below including multiplicity of $H^\cU$. Then there exists a constant $C_1$ such that
\begin{equation}
\label{e-Weyl}
E_n \ge
\frac{2\pi(1-\delta) d}{e} \,
\Big(\frac{n}{|\mathcal{U}|}\Big)^{2/d} \!\!-C_1 \qquad \text{ for all }n\in\NN .
\end{equation}
\end{lem}
The constant $C_1$ depends on the Kato norm of $V_-$ only.
Lemma \ref{l:lower} will be used in the proof of the next result, which
is a slight modification of Theorem~1 in \cite{HundertmarkKNSV-06}.

We will frequently use certain subsets of $d-1$ dimensional hyperplanes in $\RR^d$. They are unit squares in the hyperplanes.
More formally we define:

\begin{dfn}\label{d:facet}
A set $S \subset \RR^d$ is called \emph{canonical facet}
if there exists a $j \in \{1, \dots,d\}$ such that
\[
 S = \{(x_1, \dots, x_d)\mid x_j=0 \text{ and  $x_i\in[0,1]$ for $i\neq j$}    \}.
\]
A set $S$ is called a \emph{facet} if there exists a canonical facet $\tilde S$ and a vector $x \in\ZZ^d$
such that
\[
 S = x + \tilde S \coloneq \{ y \in \RR^d \mid y-x \in \tilde S\}.
\]
\end{dfn}

Let $\cU$ be an open subset of $\RR^d$, $S$ a facet as defined in Definition  \ref{d:facet},
and $\tilde \cU \coloneq \cU \setminus S$. Let $H_1$ be a Schr\"odinger operator on $\cU$, satisfying Assumption~\ref{a:regularity}.
Using quadratic forms we define $H_2$ as the Dirichlet restriction of $H_1$ to $\tilde \cU$. Then  the operators
$e^{-H_1}$, $e^{-H_2}$, and $V_{\eff}\coloneq e^{-H_2}-e^{-H_1}$ are well defined by the spectral calculus.

\begin{thm}
\label{t:singval}
(a) The operator $V_{\eff}$ is compact.\\
(b) Denote by $\mu_n$ the $n^{\text{th}}$ singular value of $V_{\eff}$ counted from
above including multiplicity.
Then there are finite positive constants $c$ and $C_2$ such that the singular values of the operator $V_{\eff}$ obey
\begin{equation}\label{e-singval}
  \mu_n \le C_2 \, e^{-c  n^{1/d}}.
\end{equation}
The constant $c$ may be chosen depending only on $d$, while $C_2$ depends only on the Kato-class
norm of $V_-$.
\end{thm}
\begin{proof}
To prove part (a), i.~e.\ that the operator $V_{\eff}$ is compact,
it is sufficient to find a  family $A_R, R >0$ of compact operators
such that the operator norm of the difference $D_R\coloneq V_{\eff}-A_R$ tends to zero
as $R\to \infty$. Indeed, such an operator family will be constructed
in the proof  of the quantitative statement (b). The operators in the family will be even trace-class.

The proof of (b) is almost the same as the one of Theorem 1 in \cite{HundertmarkKNSV-06}.
We use the same notation as there and explain only the step in the proof which is slightly different in the present setting.
Let $B\coloneq B_R\subset \RR^d$ be an open ball of radius $R>0$
containing the facet $S$ in its interior, and $S^c$ the complement of $S$.

Denote by $H_1^B$ the Dirichlet restriction of $H_1$ to the set $\cU \cap B$ and by
$H_2^B$ the Dirichlet restriction of $H_2$ to the set $\tilde \cU \cap B$. Set
\[
 D = V_\eff -\big(e^{-H_2^B}-e^{-H_1^B}\big).
\]
Let $\Erw$ and $\Prb$ denote expectation and probability for Brownian motion,
$b_t$, starting at $x$. For any open set $ U \subset \RR^d$ denote by $\tau_U=\inf\{t>0\mid b_t\not\in U\}$ the
first exit time from $U$. We use the Feynman-Kac-It\^{o} formula to express $e^{-H_1}f$ for $f\in C_c(\cU)$
as
\[
 e^{-H_1}f(x) = \Erw \l[ e^{-iS_A^1 (b)}e^{-\int_0^1 V(b_s) ds}
    \chi_{\{\tau_\cU > 1\}}(b) f(b_1) \right]
\]
where $S_A^t$ is a real valued stochastic process (It\^{o} integral) corresponding to the magnetic vector potential $A$
of the Schr\"odinger operator; this representation holds for more general $f$ \cite{Simon-79c} but a dense set suffices for our purposes.
Analogous representations hold for the operators
$e^{-H_2}$, $e^{-H_1^B}$ , and $e^{-H_2^B}$  if one replaces the condition $\chi_{\{\tau_\cU > 1\}}$
by
\[
\chi_{\{\tau_{\tilde\cU} > 1\}}, \chi_{\{\tau_{\cU\cap B} > 1\}}, \text{ and } \chi_{\{\tau_{\tilde\cU\cap B} > 1\}}, \text{ respectively.}
\]
Since the operator $D$ can be expressed in terms of four different exponentials it follows that
\[
(D f)(x)
= \Erw \l[\rho(b) \, e^{-iS_A^1 (b)}e^{-\int_0^1 V(b_s) ds} f(b_1) \right]
\]
where
\[
\rho = \chi_{\{\tau_\cU > 1\}}-
\chi_{\{\tau_{\tilde\cU} > 1\}}-  \chi_{\{\tau_{\cU\cap B} > 1\}}+ \chi_{\{\tau_{\tilde\cU\cap B} > 1\}}.
 \]
A simple transformation (using
$\chi_{\{\tau_\cU > 1\}} - \chi_{\{\tau_{\tilde\cU} > 1\}}
=\chi_{\{\tau_\cU > 1\}} (1 -\chi_{\{\tau_{\tilde\cU} > 1\}})
=\chi_{\{\tau_\cU > 1\}}  \chi_{\{\tau_{\tilde\cU} \le 1\}}
=\chi_{\{\tau_\cU > 1\}}  \chi_{\{\tau_{S^c} \le 1\}}
$
etc.) shows that
\[ \rho = \chi_{\{\tau_\cU > 1\}} \, \chi_{\{\tau_{S^c} \le  1\}} \, \chi_{\{\tau_B \le  1\}}.\]
We abbreviate $\cB \coloneq \{\tau_{S^c} \le  1\} \cap\{\tau_B \le  1\}$. The H\"older inequality implies that
\[
\vert D f\vert(x)
\le  \Big(\Erw \big[ \chi_{\{\tau_\cU > 1\}}   e^{-4\int_0^1 V(b_s) ds}\big]\Big)^{1/4}
     \Big(\Erw \big[        \chi_{\cB}(b) \big]\Big)^{1/4}
    \Big(\Erw \big[ \vert f(b_1)\vert^2 \big]\Big)^{1/2}.
\]
At this point we make the dependence of the operator $D$
on the radius of the ball $B=B_R$ explicit and denote it consequently by $D_R$.
From this point on we can follow exactly the proof of Theorem 1 in \cite{HundertmarkKNSV-06} to conclude that
\[
\|D_R\| \le \text{const}  \ \exp \left (\frac{-R^2}{32}\right).
\]
With the choice $R= R_n = n^{1/2d}$ the desired estimate \eqref{e-singval} follows.
\end{proof}

The next Lemma is an abstraction of the proof of Theorem 2 in  \cite{HundertmarkKNSV-06}, which in turn relies on \cite{HundertmarkS-02}.
The abstract formulation may be of use also in other contexts.

Let $A, B$ be two selfadjoint operators such that $V_\eff=e^{- B}-e^{-A}$  is trace class.
This implies that the sequence $\mu=\{\mu_n\}_{ n\in\NN}$ of singular values of $V_\eff$ (enumerated in decreasing order including multiplicity)
converges to zero. Let $F\colon [0,\infty)\to [0,\infty)$ be a convex function with $F(0)=0$.
In particular, $F$ is isotone because for $x\in[0,\infty),\alpha\in[0,1]$ convexity gives
$F(\alpha x)=F\big(\alpha x + (1-\alpha)0\big)
 \le \alpha F(x) + (1-\alpha) F(0) = \alpha F(x)$
which, since $F(x)\ge 0$, implies $F(\alpha x)\le F(x)$.
Set $\phi(n) = F(n) -F(n-1)$ for $ n \in \NN, n \ge 2$.

\begin{lem} \label{l:Legendre}
(a) Let $F$ be as above. Then
\begin{equation}\label{e:HS}
\int_{-\infty}^T    F\big(|\xi(\lambda,B,A)|\big) \, d\lambda \le e^T \, \langle \phi, \mu\rangle_{\ell^2(\NN)}
\end{equation}
(b)
Let $h \colon \RR \to \RR$ be a bounded measurable function with support in $(-\infty, T]$.
Then
\begin{equation}\label{e:Legendre}
\int_\RR h(\lambda) \,  \xi(\lambda,B,A) \, d\lambda
\le e^T \, \langle \phi, \mu\rangle_{\ell^2(\NN)}  + \int_{-\infty}^T G\big( |h(\lambda)|\big) \, d\lambda
 \end{equation}
where $G$ denotes the Legendre transform of $F$, i.e.\ $G(y)=\sup\{xy-F(x) \mid x\ge0\}$ for $y\ge0$.
\end{lem}
Of course the usefulness of the Lemma depends heavily on a priori information about the decay of the sequence
$\mu$. However, for the choice of operators $A= H_1$ and $B = H_2$ we do know that the singular values
decay at an almost exponential rate.

\begin{rem}
Depending on how many additional properties we assume for the function $F$,
we obtain correspondingly more information about its Legendre transform $G$.
This will be discussed next.
\begin{enumerate}
 \item The following properties hold under no additional assumptions on $F$:
$G(0)=0$, $G$ is convex (because $xy-F(x)$ is convex in $y$) and $G(y)\ge0$ for all $y$.
 \item If $\lim\limits_{x \to\infty} \frac{F(x)}{x}=\infty$ then $G$ takes on finite values only (and vice versa).
\item If $F$ is twice differentiable,  $ f \coloneq F' >0$ and $F''>0$ on $[0,\infty)$, then
\[
G(y) = \begin{cases}
        0  & \text{ if  } y \le f(0),\\
        y f^{-1} (y) -F(f^{-1} (y)) & \text{ if  } y > f(0).
       \end{cases}
\]
 Here $f^{-1}$ denotes the inverse of the function $f$.
Consequently,
\[
 \int\limits_{-\infty}^T G\big( |h(\lambda)|\big) \, d\lambda=
\int\limits_{\{\lambda \le T : \, |h(\lambda)| > f(0)\}} \big(|h(\lambda)| f^{-1}(|h(\lambda)|)- F(f^{-1} (|h(\lambda)|))\big)\, d\lambda .
\]
\item If there exist a positive constant $C$ and an exponent $p>1$ such that $\mu_n \le C n^{-p}$, then one can choose the function $F$
as $F(x) = x^{q +1}$, where $q $ is any number smaller than $p-1$. Indeed, for this choice of $F$ we have
$\phi(n)  \le  (q+1) n^{q}$. Thus
\[
\sum_n \mu_n \phi(n) \le (q+1) C \sum_n n^{-p} n^{q} < \infty.
\]
Note that in this case $G(y)=q\left(\frac{y}{q+1}\right)^{\frac{q+1}{q}}$.

\item
If there exist positive constants $c,C$ and $p$ such that $\mu_n \le C e^{-cn^p}$ for all $n \in \NN$, then
for each value of $t < c$, the choice
\[
 F(x) = \int_0^x \big( e^{t y^p} -1\big) \, dy
\]
gives a finite right hand side in \eqref{e:HS}. Indeed, $\phi(n) = \int_{n-1}^n \big( e^{t y^p} -1\big) \, dy\le e^{t n^p}$
and thus
\[
  \langle \phi, \mu\rangle_{\ell^2(\NN)}  \le C \sum_n  e^{-cn^p} e^{tn^p} < \infty.
\]
The Legendre transform for such a choice of $F$ satisfies
\[
G(y)
\le yf^{-1}(y) =  \left ( \frac{\log(1+y)}{t} \right)^{1/p}  \text{ for all } y \ge 0.
\]
Thus in this specific case inequality \eqref{e:Legendre} reads
\begin{multline}\label{e:Legendre-Exponential}
\int_\RR h(\lambda) \,  \xi(\lambda,B,A) \, d\lambda
\le e^T \, C\sum_{n\in\NN} e^{-(c-t) n^p}+ \\
\int |h(\lambda)| \left( \frac{\log(1+\vert h(\lambda)\vert)}{t} \right)^{1/p} \, d\lambda
\end{multline}
which recovers the result of \cite{HundertmarkKNSV-06}.
\end{enumerate}
\end{rem}

Now we prove Lemma \ref{l:Legendre}.
\begin{proof}
Since $V_\eff$ is trace class but not necessarily the operator difference $A-B$, the
SSF is defined via the invariance principle
\begin{displaymath}
\int_{-\infty}^T F(|\xi(\lambda, B, A)|)\, d\lambda
=
\int_{-\infty}^T F(|\xi(e^{-\lambda}, e^{-B}, e^{-A})|)\, d\lambda .
\end{displaymath}
Now a change of variables gives us
\begin{displaymath}
\int_{-\infty}^T F(|\xi(e^{-\lambda}, e^{-B}, e^{-A})|)\, d\lambda
\le
e^T \int_{e^{-T}}^\infty F(|\xi(s , e^{-B}, e^{-A})|)\, d s .
\end{displaymath}
To the last expression we can apply the estimate of \cite{HundertmarkS-02} and bound it above by
\[
\int_{e^{-T}}^\infty F(|\xi(s , e^{-B}, e^{-A})|)\, d s
\le
\sum_{n \in \NN} \mu_n(V_{\eff}) \phi(n).
\]
This establishes claim (a). To prove (b), we note that by the very definition of the Legendre transform,
the Young inequality
\[
|h\cdot\xi  | \le F(|\xi|) + G(|h|)
\]
holds. Integrating over $\lambda$ we obtain
\begin{equation}
\label{e:Young}\nonumber
\int h(\lambda) \xi(\lambda)\, d\lambda  \le
\int_{-\infty}^T F(|\xi(\lambda, B,A)|) \, d\lambda + \int G(|h(\lambda)|) \, d\lambda.
\end{equation}
Together with (a) this completes the proof.
\end{proof}

\section{Almost additivity for the eigenvalue counting functions}\label{s:additivity}

In this section we prove Theorem \ref{t:IDS-colouring}.
For this aim we will apply a Banach space-valued ergodic theorem
obtained in \cite{LenzMV-08}. Actually, for our purposes it will be convenient to quote a slightly streamlined version of this result from \cite{LenzSV}.
To spell it out we need to introduce the notion of a \emph{boundary term} and the properties of \emph{almost additivity} and \emph{invariance}.

\begin{dfn}\label{def_boundaryterm}
A function $b\colon \FinSetsZd  \rightarrow [0,\infty)$ is called a \emph{boundary term} if the following three properties hold:
 \begin{enumerate}[(i)]
   \item $b(Q)=b(Q+x)$ for all $x\in \ZZ^d$ and all $Q\in \FinSetsZd $,
   \item $\lim\limits_{j\rightarrow\infty} \frac{b(U_j)}{\numb{U_j}}=0$ for any van Hove sequence $(U_j)_{j\in \NN}$ and
   \item there exists a constant $D\in (0,\infty)$ such that
$$b(Q)\leq D\numb{Q} \quad \text{ for all } \quad Q\in \FinSetsZd $$
 \end{enumerate}
\end{dfn}
\begin{dfn}
  Let $(X,\|\cdot\|)$ be a Banach space and $F$ a function $
  \FinSetsZd \to	 X$.

  (a) The function $F$ is said to be \emph{almost-additive} if there
  exists a boundary term $b$ such that
  \begin{equation*}
    \| F ( \cup_{k=1}^m Q_k ) - \sum_{k=1}^m F(Q_k)\| \leq \sum_{k=1}^m b
    (Q_k)
  \end{equation*}
  for all $m\in\mathbb{N}$ and all pairwise disjoint sets $Q_k\in\FinSetsZd$,
  $k=1,\ldots,m$.

  (b) Let $\cC :\ZZ^d \longrightarrow \cA$ be a colouring.
  The function $F$ is said to be \emph{$\cC$-invariant} if
  \begin{equation*}
    F( Q) = F (Q+x)
  \end{equation*}
  whenever $x \in \ZZ^d$  and $Q\in\FinSetsZd$ obey $\cC\vert_Q +x= \cC\vert_{Q+x}$.

  In this case there exists a function $\widetilde F$ defined on the (classes of) patterns
  such that $\widetilde F(P)=F(Q)$ if $\cC\vert_Q = P$.
\end{dfn}

If $F\colon \FinSetsZd \rightarrow X$ is  almost additive and invariant there exists a $K \in (0,\infty)$ such that
\begin{equation}
\| F(Q) \|   \le K\numb{Q} \hspace{1cm}\text{ for all } Q \in  \FinSetsZd.
\end{equation}

Now we are in the position to quote the Banach space valued ergodic theorem from \cite{LenzMV-08}, see \cite[\S 5.1]{LenzSV} as well.

\begin{thm}
 \label{t:ergodic-theorem}
Let $\mathcal A$ be a finite set of colours, ${\mathcal C}\colon \ZZ^d \rightarrow {\mathcal A}$
a colouring and $(U_j)$ a van Hove sequence along which the frequencies of all patterns $P\in \bigcup_{M\in \NN} {\mathcal P}(C_M)$ exist.
Let $F: {\mathcal F}(\ZZ^d)\rightarrow X$  be a ${\mathcal C}$-invariant and almost-additive function.
Then the limit
\[
\overline F \coloneq \lim\limits_{j\rightarrow \infty}\frac{F(U_j)}{\numb{U_j}}
 = \lim_{M\rightarrow\infty} \sum\limits_{P\in {\mathcal P}(C_M)} \nu_P \frac{\widetilde F(P)}{\numb{C_M}}
\]
exists in $X$. Furthermore, for $j,M\in \mathbb N$ the bound
\begin{equation}
\Big\Vert \overline F  -\frac{F(U_j)}{\numb{U_j}} \Big\Vert
\le
 2\frac{b(C_M)}{M^d} + (K+D)\frac{\numb{\partial^{M} U_j}}{\numb{U_j}}+K\sum\limits_{P\in {\mathcal P}(C_M)} \Big \vert \frac{\numb[P]{ ({\mathcal C}\vert_{U_j})}}{\numb{U_j}}- \nu_P  \Big \vert
\end{equation}
holds.
\end{thm}

We want to apply the ergodic theorem
to the eigenvalue counting functions of Schr\"odinger operators,
considered as elements of $X\coloneq L^p(I)$ for a fixed finite
interval $I\subset \RR$ and $p\in[1,\infty)$.

More precisely, we study the function
\[
F: \FinSetsZd \to L^p(I), \quad Q \mapsto N\big( \cdot,  H^Q \big)
\]
with $N\big(\lambda, H^Q\big) \coloneq
\Tr \chi_{(-\infty,\lambda]} \big( H^Q \big)$
for $\lambda\in\RR$.
Note that this notation is slightly different from the one used in Theorem \ref{t:IDS-colouring}.
The reason is that in the proofs we use  a more general class of operators than which was necessary to formulate the main result
and thus need a bit more flexibility.
To conclude Theorem \ref{t:IDS-colouring}
we need to show that $F$ fulfils the hypotheses of Theorem \ref{t:ergodic-theorem}.
This is done in the following
\begin{lem}\label{l:almost-additivity}
The function $ F: \FinSetsZd \to L^p(I)$ is invariant and almost-additive.
\end{lem}
\begin{proof}
Note that $H$ is a local operator, thus $H^{\cU}$ depends only on
$A\vert_{\cU}$ and $V\vert_{\cU}$, for any open $\cU\subset \RR^d$.
The translation operator $T_yf(x) = f(x-y) $ is unitary for any $y\in \RR^d$.
Thus the spectrum of $H$, resp.\ $H^\cU$,
is invariant under conjugation by $T_y$. It follows that the function $F$ is $\cC$-invariant.

To prove almost-additivity, let $Q=\bigcup_{k=1}^m Q_k$ with pairwise disjoint
$Q_k\in\FinSetsZd ,k=1,\dots,m$. We need to compare $F(Q)$ to
$\sum_{k=1}^mF(Q_k)$. Note that $W_Q=\bigcup_{k=1}^m W_{Q_k}$
but the $W_{Q_k}$ need not be pairwise disjoint because their boundaries can
touch.
There are two extreme cases:
\begin{enumerate}
\item All $W_{Q_k}$ are pairwise disjoint. Then
$\interior{W_Q}=\bigcup_{k=1}^m \interior{W_{Q_k}}$, and consequently
$H^Q=\bigoplus_{k=1}^m H^{Q_k}$ and $F(Q)=\sum_{k=1}^m F(Q_k)$.

\item No $W_{Q_k}$ is disjoint from all others. Since $\partial W_{Q_k}$
consists of at most $2d\numb{\partial Q_k}$ facets (where $\partial Q_k$ denotes the
combinatorial boundary of $Q_k\subset\ZZ^d$)
the sets $\interior{W_Q}$ and $\bigcup_{k=1}^m \interior{W_{Q_k}}$
differ by at most $2d\sum_{k=1}^m\numb{\partial Q_k}$ facets.
\end{enumerate}
In fact, the latter case gives an upper bound for the general case (where the
``isolated'' $Q_k$ simply do not contribute), and we can drop the factor 2
because touching facets need to be counted once only (they are counted twice in
the sum).

So, let $M\le d\sum_{k=1}^m\numb{\partial Q_k}$ be the number of facets by which $\interior{W_Q}$ and
$\bigcup_{k=1}^m \interior{W_{Q_k}}$ differ and enumerate them arbitrarily as
$S_1,\dots,S_M$. Set $H_0\coloneq H^Q$ and
\[ H_j\coloneq H^{U_j} \text{ where } U_j \coloneq W_Q \setminus
\bigcup_{i=1}^j S_i \]
for $j=1,\dots,M$.
Clearly, $H_M=\bigoplus_{k=1}^m H^{Q_k}$ and $N(\cdot, H_M)= \sum_{k=1}^m F(Q_k)$.
This relation allows us to write the difference that we want to estimate as a sum of SSFs:
\begin{align*}
F(Q) - \sum_{k=1}^mF(Q_k) &= \sum_{j=1}^M \big( N(\cdot, H_{j-1}) - N(\cdot, H_j) \big) 
	= \sum_{j=1}^M \xi(\cdot, H_{j-1}, H_j)
\end{align*}
$H_{j-1}$ and $H_j$ differ exactly by a Dirichlet condition at one facet $S_j$
so that Theorem~\ref{t:singval} applies and gives the estimate $\mu_n\leq
Ce^{-cn^{1/d}}$ for the singular values.
Now we can apply Lemma~\ref{l:Legendre} with $T\coloneq\sup I$,
$A\coloneq H_j$, $B\coloneq H_{j-1}$ and $F(x)\coloneq x^p$.
Then inequality~\eqref{e:HS} reads
\begin{equation}\label{e:SSF-bound}
\int_I |\xi(\cdot, H_{j-1}, H_j)|^p \leq e^T\sum_n C_2 p n^{p-1}e^{-cn^{1/d}}
\eqcolon \tilde C^p
\end{equation}
with a constant $\tilde C$ which is independent of $j$.
The triangle inequality thus gives
\begin{align*}
\left\| F(Q) - \sum_{k=1}^mF(Q_k) \right\|_{L^p(I)} &\le M \tilde C
      \le d \, \tilde C \sum_{k=1}^m \numb{\partial Q_k}
      \eqcolon \sum_{k=1}^m b(Q_k).
\end{align*}
The function $b:\FinSetsZd \to\RR,Q\mapsto d \, \tilde C \numb{\partial Q}$
satisfies the three conditions required for a boundary term.
\end{proof}
\begin{proof}[Proof of Theorem~\ref{t:IDS-colouring}]
Given the previous Lemma, we can apply the abstract ergodic theorem to our counting
functions associated to finite subsets $Q$. In order to check the form of
the error estimate, we note that the constant $\tilde C$ in equation \eqref{e:SSF-bound}  satisfies
\[
\tilde C^p = C_2 \,  e^T \const(p,d)
\]
with $C_2$ being the constant from Theorem \ref{t:singval} depending on the Kato-norm of $V_{\cC,-}$. Also,
\[
b(Q) = d \tilde C \numb{\partial Q}
\]
so that the bound $D$ on $b$ in the general abstract ergodic theorem can be taken to be
$d\tilde C$. It follows that  $\frac{2b(C_M)}{M^d}\leq 4d^2\tilde C/M$.

Finally, the uniform lower estimate on the eigenvalues established in Lemma~\ref{l:lower}
gives a uniform upper estimate on the number of eigenvalues in $(-\infty,T]$, namely
\[
 (T+C_1)^{d/2}  \, \left(\frac{e}{2\pi (1-\delta)}d)\right )^{d/2}  \, |\cU|.
\]
Hence we can take
\[
K =C_3 \, (T+C_1)^{\frac d2}
\]
as the uniform bound on the function $F$ in the abstract ergodic theorem.
Here the constants $C_1,C_2,C_3$  depend only on $d$ and the Kato norm of $V_{\cC,-}$.
\end{proof}

\section{Application to random operators }\label{s:random}
In order to apply our results to random operators we quote the necessary random versions
of the definition and main theorem from \cite{GruberLV-08}:
Let  $(\Omega,\PP)$ be a probability space such that $\ZZ^d$
acts ergodically on $(\Omega,\PP)$.  We denote the $\ZZ^d$-action on $\Omega$ by
$ x\colon \omega\mapsto\omega -x$.
A random $\cA$-colouring is a map
\[
 \cC : \Omega\longrightarrow \bigotimes_{\ZZ^d} \cA \quad\text{with}\quad
\cC(\omega-y)_{x-y} = \cC(\omega)_{x}
\]
for all $x,y \in\ZZ^d$. Note that for each fixed $\omega$ we obtain a (usual) $\cA$-colouring.
By the (usual) ergodic theorem for scalar functions the frequencies of patterns exist almost surely.
Thus we can apply our  abstract Banach space valued ergodic theorem:

\begin{thm} \label{t:ergodictheorem-random}
 Let $\cA$ be a finite set, $\cC$ be a random $\cA$-colouring and
  $(X,\|\cdot\|)$ a Banach space.  Let $(U_j)_{j\in\mathbb{N}}$ be a van Hove
  sequence.
  For each fixed $\omega\in \Omega$ let $F_\omega : \FinSetsZd
  \longrightarrow X$ be a $\cC(\omega)$-invariant, almost-additive
  bounded function. Assume that the family $(F_\omega )_{\omega\in\Omega}$ is $\ZZ^d$-homogeneous, i.e.\ $F_{\omega+x}(Q+x)=F_\omega(Q)$ for all $x\in\ZZ^d,Q\in\FinSetsZd$.
  Then, for almost every $\omega\in\Omega$ the limits
  \begin{equation*}
    \overline{F}_\omega \coloneq  \lim_{j\to\infty} \frac{F_\omega(U_j)}{|U_{j}|}    = \lim_{M\to\infty}
    \sum_{P \in {\mathcal P}(C_M)} \nu_P \frac{\widetilde{F}(P)}{|C_M|}
  \end{equation*}
  exist in the topology of $(X,\|\cdot\|)$ and are equal. In particular, $\overline{F}_\omega$ is almost surely independent of $\omega$.
\end{thm}

Concretely, we apply this to colourings given by local models for the potentials like in Section~\ref{s:results},
but now with a randomly chosen colouring so that all operators and counting functions additionally depend on the random
variable $\omega$.  For the formulation of the theorem we introduce a distribution function $N \colon \RR \to \RR$ defined by a trace per unit volume formula
(sometimes called Pastur-Shubin formula)
\begin{equation}
N(\lambda) \coloneq \int_\Omega \Tr \left[\chi_{W_0} \chi_{(-\infty,\lambda]}(H_\omega) \right] d\mathbb{P}(\omega).
\label{e:Shubin-Pastur}
\end{equation}

By applying the random version of the ergodic theorem we get the random version of Theorem~\ref{t:IDS-colouring}:

\begin{thm}\label{t:IDS-random}
Let $\cC : \Omega\longrightarrow \bigotimes_{\ZZ^d} \cA $ be a random colouring,  $(U_j)_{j\in \mathbb N}$ a van Hove sequence.
Let $I\subset \RR$ be a finite interval and $p \in [1,\infty)$. Then for $j\to \infty$ we have for almost all $\omega\in\Omega$
\begin{equation}
 \int_I \left|N\big(\lambda) - \frac{1}{\numb{U_j}}  N_\omega\big(\lambda, {U_j}\big) \right|^p \, d\lambda\to 0. \label{e:IDS-random}
\end{equation}

\end{thm}
\begin{rem}
Of course there are similar error estimates as in Theorem~\ref{t:IDS-colouring}.
\end{rem}
\begin{proof}
The existence of the limit in Equation~\eqref{e:IDS-random} follows directly from our abstract theorem since we checked all requirements
in the proof of Theorem~\ref{t:IDS-colouring} already.

For the proof of the Shubin-Pastur formula~\eqref{e:Shubin-Pastur} we use a variation of the proof
of Theorem~3 in \cite{GruberLV-07}:
First notice that
\[
	N(\lambda) =  \frac{1}{\numb{U_j}} \int_\Omega \Tr \left[\chi_{W^{U_j}} \chi_{(-\infty,\lambda]}(H_\omega) \right] d\mathbb{P}(\omega)
\]
independently of $U_j$ due to additivity and invariance.
Now it suffices to show that
\begin{equation}
\frac1{\numb{U_j}}  \Tr \left[\chi_{W^{U_j}} e^{-tH_\omega} - e^{-tH_\omega^{U_j}}\right] \to 0,\quad j\to\infty
\tag{*}\label{e:star}
\end{equation}
for all $t$. 
We use the abbreviation $ \oplus H_\omega = H_\omega^{U_j}\oplus H_\omega^{\ZZ^d\setminus U_j}$ 
and the  linearity of the trace to conclude
\begin{multline*}
\Tr \left[\chi_{W^{U_j}} e^{-tH_\omega} - e^{-tH_\omega^{U_j}}\right]  
\\=
\Tr \left[\chi_{W^{U_j}} (e^{-tH_\omega} - e^{-t \oplus H_\omega}) \right]  
+
\Tr \left[\chi_{W^{U_j}} e^{-t \oplus H_\omega}-     e^{-tH_\omega^{U_j}}\right]  .
\end{multline*}
The second term actually vanishes. Thus it suffices to estimate the first term, which we do next.
Let us note that for a compact operator $K$ and a bounded operator $B$ on the same Hilbert space
the corresponding singular values obey the relation 
\[
       \mu_n(BK) \le \|B\| \mu_n(K)  .
\]
In particular this holds for $K=V_{\eff}$ as in Theorem ~\ref{t:singval} 
and $B=\chi_{W^{U_j}}$.
Now we can proceed exactly as in Lemma~\ref{l:almost-additivity},
using that the addition of the boundary conditions by which $\oplus H_\omega$
differs from $H_\omega$ gives a boundary term in $L^p(I)$ for the corresponding SSF.
Now the claim \eqref{e:star} follows from the van Hove property.
\end{proof}

\renewcommand{\MR}[1]{} 
\bibliographystyle{amsalpha}
\bibliography{glv3}

\end{document}